\documentclass{article}      

\usepackage{arxiv}

\usepackage[utf8]{inputenc}
\usepackage[T1]{fontenc}
\usepackage{hyperref}
\usepackage{url} 
\usepackage{amsmath, amsthm, amssymb}
\usepackage{dsfont}
\usepackage{graphicx}
\usepackage{booktabs, enumitem}
\usepackage{float}
\usepackage{doi}
\usepackage[table, dvipsnames,svgnames]{xcolor}
\usepackage[style=numeric-comp ,sorting=nyt, url=true, giveninits=true]{biblatex}

\let\oldrmdefault\rmdefault
\usepackage{LobsterTwo}
\let\rmdefault\oldrmdefault
\undef{\textlozenge}
\usepackage{gfsartemisia}

\newcommand*\e[1]{\ensuremath{\text{\LobsterTwo #1}}}
\renewcommand{\mathcal}[1]{{\e #1}}

\newtheorem{theorem}{Theorem}[section]

\title{Study via Approximate Symmetries of the Emergence of Secondary Flux in the Bevilacqua-Galeão Model}

\date{\today}

\author{\href{https://orcid.org/0000-0002-9216-3864}{\includegraphics[scale=0.06]{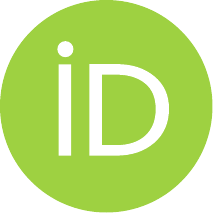}\hspace{1mm}Yuri~Bozhkov}\\
	Institute of Mathematics, Statistics and Scientific Computing\\
	University of Campinas -- UNICAMP\\
	Campinas--SP, Brazil \\
	\texttt{bozhkov@unicamp.br} \\
	\And
	\href{https://orcid.org/0000-0003-2928-2809}{\includegraphics[scale=0.06]{orcid.pdf}\hspace{1mm}Stylianos~Dimas} \\
	Department of Mathematics, Division of Fundamental Sciences\\
	Technological Institute of Aeronautics -- ITA\\
	São José dos Campos--SP, Brazil \\
	\texttt{stylianos.dimas@gp.ita.br} \\
	\And
	\href{https://orcid.org/0000-0002-9616-6093}{\includegraphics[scale=0.06]{orcid.pdf}\hspace{1mm}Antônio J.~Silva Neto} \\
	Polytechnic Institute\\
	Rio de Janeiro State University -- UERJ\\
	Nova Friburgo--RJ, Brazil\\
	\texttt{ajsneto@iprj.uerj.br}
}

\renewcommand{\headeright}{}

\hypersetup{
pdftitle={Study via Approximate Symmetries of the Emergence of Secondary Flux in the Bevilacqua-Galeão Model},
pdfsubject={math.AP},
pdfauthor={Yuri~Bozhkov, Stylianos~Dimas, Antônio J.~Silva Neto},
pdfkeywords={First keyword, Second keyword, More},
}

\begin{document}
\maketitle

\begin{abstract}
	The Bevilacqua-Galeão Model of Anomalous Diffusion introduces two fluxes: a primary flux that follows Fick's law of diffusion, representing the fraction of particles undergoing classical diffusion, and a secondary flux modeled by a fourth-order differential term, which accounts for retention phenomena. We investigate the ``analytic emergence'' of this secondary flux by treating the model as a (singular) perturbation of the heat equation, which describes the classical diffusion. Rather than applying traditional perturbation methods, or a straightforward Fourier transformation, we employ the powerful framework of enhanced modern group analysis to study this problem. Specifically, by utilizing approximate symmetries we derive an analytic expression for the emergent secondary diffusion as a perturbation of the classical diffusion process, the approximate fundamental solution and the approximate solution of the related Cauchy problem.
\end{abstract}


\section{Introduction}

Bevilacqua, Galeão and Costa introduced the following PDE in \cite{BevGalCos2011a}
\begin{equation}\label{eq:GoverningPDE3+1}
	p_t(\mathbf x,t) = \nabla\cdot\left(\beta\mathbf D\cdot\nabla p(\mathbf x,t)\right)-\nabla\cdot\left(\beta(1-\beta )\mathbf R\nabla\cdot\left(\nabla\left(\mathbf C\cdot\nabla p(\mathbf x,t)\right)\right)\right),
\end{equation}
where $p$ represents the concentration of a substance at position $\mathbf x\in\mathds R^3$ and time $t$ and $\beta, \mathbf C, \mathbf D, \mathbf R$ are parameters that may depend on the independent and dependent variables. Their goal was to derive the governing equation for \emph{diffusion with retention} phenomena. In this context, the diagonal matrices $\mathbf C, \mathbf D, \mathbf R$, which correspond to the physical components of the model, will be determined by the experimental data, while $0\le\beta\le1$ --- called \emph{redistribution parameter} --- controls the magnitude of the retention effect with $\beta=0$ corresponding to the static model and $\beta=1$ to the classical diffusion-convection model.

Equation~\eqref{eq:GoverningPDE3+1} is an example of a model that describes \emph{anomalous diffusion processes}. Anomalous diffusion processes, which deviate from classical Fickian diffusion, are observed in various complex systems where particle transport is influenced by retention, trapping, or the presence of heterogeneous media. To capture these phenomena, Bevilacqua, Galeão  and Costa extended the traditional diffusion equations by incorporating retention effects through a bimodal flux distribution, that is, by the inclusion of a fourth-order spatial derivative which precisely accounts for the anomalous behaviour. This kind of approach provides a unified framework capable of describing diffusion with partial hold-up or delayed transport, which is relevant in fields ranging from porous media flow to population dynamics and chemical reactions. A few key applications are:
\begin{itemize}
\item Material Science: Analyzing the diffusion of particles or molecules through complex materials with heterogeneous structures.
\item Chemical Engineering: Describing diffusion processes in chemical reactions where substances are temporarily trapped or retained in certain phases.
\item Environmental Science: Studying pollutant transport in porous media, such as soil and groundwater, where retention effects are significant.
\item Population Dynamics: Modeling the spread of populations in environments where individuals may experience retention or delayed movement.
\item Biomedical Applications: Modeling drug delivery systems where active ingredients diffuse through biological tissues with retention properties.
\item Finance: Flux of Capital.
\end{itemize}

It is no surprise that this model has been the subject of extensive study --- particularly through numerical methods --- with special emphasis on the inverse problems associated with parameter estimation and spatial variability, \cite{BevGalCos2011b,BevJiaSil2016a,JiaBevNet2018a}. 

The $1+1$-dimensional version of \eqref{eq:GoverningPDE3+1} assumes the form
\begin{equation}\label{eq:GoverningPDE1+1}
	p_t(x,t) = \frac{\partial}{\partial x}\left(\beta d \frac{\partial p(x,t)}{\partial x}\right) - \frac{\partial}{\partial x}\left[\beta(1-\beta)r\frac{\partial^2 }{\partial x^2}\left(c\frac{\partial p(x,t)}{\partial x}\right)\right].
\end{equation}
When all the parameters are constants, we arrive at the simplest version of the model,
\begin{equation}\label{eq:SimplestGoverningPDE1+1}
	p_t(x,t) = \beta d p_{xx}(x,t) - \beta(1-\beta)rc  p_{xxxx}(x,t).
\end{equation}
In this case, it is easy to identify $d>0$ as the \emph{classical diffusion coefficient} of the harmonic term associated to the primary flux and $rc>0$ as the diffusion coefficient of the biharmonic term controlling the mass transfer time rate associated to the secondary flux. This reduced equation concisely expresses the fundamental dynamics of the elaborate models \eqref{eq:GoverningPDE3+1} and \eqref{eq:GoverningPDE1+1}, thus permitting an undemanding preliminary analysis and resolution in various contexts.

By the relatively simple change of scale,
\begin{align}\label{eq:ChangeOfScale1}
	\bar{x}= \sqrt{\frac{d}{1-\beta}}\ x,&&  \bar{t} =\frac{d^2 \beta}{1-\beta}t, && u(\bar{x},\bar{t})=p(x,t),
\end{align}
from equation~\eqref{eq:SimplestGoverningPDE1+1} we arrive, after dropping the bars, to the PDE
\begin{equation}\label{eq:PerturbedEq}
	u_t=u_{xx} - \varepsilon\,u_{xxxx},
\end{equation}
where $\varepsilon=rc$.

The objective of this work is to express analytically the secondary flux due to retention by considering that $\varepsilon\ll1$, that is, by considering equation \eqref{eq:PerturbedEq} as a (singular) perturbation of the heat equation,
\begin{equation}\label{eq:Heat}
	u_t=u_{xx}.
\end{equation}
Specifically, we shall look for the \emph{fundamental approximate solution} of~\eqref{eq:PerturbedEq}, which will be of the form
$$
	\upsilon_h(x,t)+\varepsilon\,\upsilon_\varepsilon(x,t),
$$
hence separating the two fluxes as desired.

For this study we shall not use standard perturbation methods or the Fourier transformation but  techniques from the \emph{enhanced modern group analysis}. Enhanced modern group analysis is a mathematical framework that extends Sophus Lie's notion of a Lie point symmetry, and its algorithmic way of obtaining them, to different contexts, thus enlarging its usefulness. In our case, the main method used will be the \emph{approximate symmetries} --- first conceived and developed by Baikov, Gazizov, and Ibragimov, \cite{Ibr2008a} --- intended for studying the symmetry properties of differential equations that are perturbed by small parameters. 

We choose equation~\eqref{eq:PerturbedEq} because it is easy to verify that our  method will give the same result found using the Fourier transformation. To solidify our proposed method we go one step forward: we consider the model that corresponds to $rc=\varepsilon x^2$, that is, the PDE
\begin{equation}\label{eq:PerturbedEq2}
	u_t=u_{xx} - \varepsilon\,x^2 u_{xxxx},
\end{equation}
a model that although remains linear becomes quite challenging for the Fourier transform.

The rest of the article is structured as follows, in section $2$ we give a brief account of the different methods from modern group analysis employed --- adapted to the equation at hand. Section $3$ contains our main results, that is, the fundamental approximate solution for equations~\eqref{eq:PerturbedEq} and \eqref{eq:PerturbedEq2} and their approximate solutions for any Cauchy problem that involves them. We dedicate Section $4$ to the comparison of two particular Cauchy problems with their respective numerical solutions. Finally, in Section $5$,  we present our concluding thoughts on the methodology  and possible future research directions.

\section{Methodology}

In  the subsections that follow we give a brief account of the main tools used throughout this work by exposing the relevant parts of the symmetry theory suitably adapted for the two equations, \eqref{eq:Heat} and \eqref{eq:PerturbedEq}, under study.

\subsection{Lie Point Symmetries}\label{sec:LiPointSymmetries}

In the heart of our analysis resides the concept of \emph{symmetry}. Symmetry, loosely put, is a transformation --- a \emph{diffeomorphism}, to be precise --- between elements of a differential equation that leaves it invariant. These symmetries form a \emph{group}, and for our purposes we will restrict ourselves to its connected component, which constitutes a \emph{Lie group}. Practically, this means that the kind of symmetries that we shall use will depend on a continuous variable, $\epsilon$, and we shall identify the symmetry that corresponds to $\epsilon=0$ with the identity transformation --- which is the (trivial) symmetry that any differential equation admits. But this is not the only reason for making this restriction: we can now elegantly represent any such symmetry with an element of a \emph{Lie algebra}, that is, with a special type of \emph{vector space}.

There is one more restriction to make: we shall assume that the transformations, our symmetries, involve only the independent and dependent variables. That is, the group of one parameter transformations, for the study of \eqref{eq:Heat}, takes the form
\begin{align*}
	\bar x =& \bar x(x,t,u;\epsilon),\\
	\bar t =& \bar t(x,t,u;\epsilon),\\
	\bar u =& \bar u(x,t,u;\epsilon).
\end{align*}
In other words, we shall deal with \emph{Lie point symmetries}, and going to the Lie algebra level, we shall represent these symmetries with vectors of the form
\begin{equation}\label{eq:LiePointSymmetry}
	\mathcal{X} = \xi^1(x,t,u)\frac{\partial }{\partial x}+\xi^2(x,t,u)\frac{\partial }{\partial t} + \eta(x,t,u) \frac{\partial }{\partial u}. 
\end{equation}
where
\begin{align*}
	 \xi^1(x,t,u) = \left.\frac{\partial}{\partial\epsilon}\bar x(x,t,u;\epsilon)\right\rvert_{\epsilon=0}, && \xi^2(x,t,u) = \left.\frac{\partial}{\partial\epsilon}\bar t(x,t,u;\epsilon)\right\rvert_{\epsilon=0}
   \intertext{and} 
   \eta(x,t,u) = \left.\frac{\partial}{\partial\epsilon}\bar u(x,t,u;\epsilon)\right\rvert_{\epsilon=0}.
\end{align*}
The inverse direction is also --- at least locally --- true. Having a symmetry as an element of a Lie algebra of the form \eqref{eq:LiePointSymmetry} one can retrieve the corresponding continuous transformation as an element of a (local) Lie group by solving the following Cauchy problem
$$
	\begin{cases} 
		\dfrac{\mathrm{d}}{\mathrm{d}\epsilon}\bar x(\epsilon) = \xi^1(\bar x,\bar t, \bar u),& \bar x(0) = x,\\[10pt]
		\dfrac{\mathrm{d}}{\mathrm{d}\epsilon}\bar t(\epsilon) = \xi^2(\bar x,\bar t,\bar u),& \bar t(0) = t,\\[10pt]\dfrac{\mathrm{d}}{\mathrm{d}\epsilon}\bar u(\epsilon) = \eta(\bar x,\bar t, \bar u),& \bar u(0) = u,
	\end{cases}
$$
for this reason $Mathcal X$ is also called \emph{infinitesimal generator}. We call this procedure \emph{exponentiation}.

But how we shall obtain the symmetries of a differential equation? This is where the ingenious idea of Sophus Lie comes into the light. Since a Lie algebra is a \emph{linearization} of a Lie group around the identity element --- the tangent space at the identity --- the same happens for the underlining problem. That is, instead of the messy and nonlinear problem of finding transformations that keep invariant a differential equation we can \emph{linearize} the task and look for vectors $\mathcal X$ that its coefficients satisfy a set of linear differential equations. Indeed, under certain conditions that the differential equation must satisfy --- and for the cases that we shall study there are indeed satisfied, the vector $\mathcal X$ is a symmetry of \eqref{eq:Heat} if, and only if, it satisfies the elegant and linear symmetry condition:
\begin{equation}\label{eq:LinSymCond}
	\mathcal{X}^{(2)}[u_t-u_{xx}]\Big\rvert_{u_t=u_{xx}} =0, 
\end{equation}
where $\mathcal{X}^{(2)}$ is the second \emph{prolongation} of the vector field $\mathcal X$,
\begin{equation}\label{eq:SecondPrologation}
	\mathcal{X}^{(2)} = \mathcal{X} + \eta^{(1,0)}(x,t,u) \frac{\partial }{\partial u_x}+ \eta^{(0,1)}(x,t,u) \frac{\partial }{\partial u_t}+ \eta^{(2,0)}(x,t,u) \frac{\partial }{\partial u_{xx}} + \eta^{(1,1)}(x,t,u) \frac{\partial }{\partial u_{xt}}+ \eta^{(0,2)}(x,t,u) \frac{\partial }{\partial u_{tt}},
\end{equation}
and
\begin{align*}
	\eta^{(1,0)}(x,t,u) =& D_x\left(\eta(x,t,u)-\xi^1(x,t,u)u_x-\xi^2(x,t,u)u_t \right) +\xi^1(x,t,u)u_{xx}+\xi^2(x,t,u)u_{xt},\\ 
	\eta^{(0,1)}(x,t,u) =& D_t\left(\eta(x,t,u)-\xi^1(x,t,u)u_x-\xi^2(x,t,u)u_t \right) +\xi^1(x,t,u)u_{tx}+\xi^2(x,t,u)u_{tt},\\ 
	\eta^{(2,0)}(x,t,u) =& D_x^2\left(\eta(x,t,u)-\xi^1(x,t,u)u_x-\xi^2(x,t,u)u_t \right) +\xi^1(x,t,u)u_{xxx}+\xi^2(x,t,u)u_{xxt},\\ 
	\eta^{(1,1)}(x,t,u) =& D_xD_t\left(\eta(x,t,u)-\xi^1(x,t,u)u_x-\xi^2(x,t,u)u_t \right) +\xi^1(x,t,u)u_{xxt}+\xi^2(x,t,u)u_{xtt},\\
	\eta^{(0,2)}(x,t,u) =& D_t^2\left(\eta(x,t,u)-\xi^1(x,t,u)u_x-\xi^2(x,t,u)u_t \right) +\xi^1(x,t,u)u_{xtt}+\xi^2(x,t,u)u_{ttt},
\end{align*}
with $D_x, D_t$ denoting the total derivative with respect to $x$ and $t$, respectively. 

Considering the linear symmetry condition\eqref{eq:LinSymCond} as a multivariable polynomial of the form $\mathcal{P}(u_x, u_t, u_{xx}, u_{xt}, u_{tt})$, \eqref{eq:LinSymCond} breaks down to an overdetermined --- in our case --- system of (linear) partial differential equations which is called the \emph{determining equations}. Its general solution will provide us the Lie algebra of Lie point symmetries of \eqref{eq:Heat}, $\mathcal L$\ . The process is quite algorithmic and and usually keyed as the \emph{Lie algorithm}. 

It is a very well-known fact, \cite[p.~118]{Olver2k}, that the infinite dimensional Lie algebra of Lie point symmetries of the heat equation \eqref{eq:Heat} is spanned by the vectors, 
\begin{equation}\label{Heat:LieAlgebra}
\begin{split}
	\mathcal{X}_1 =& \partial_x,\\
	\mathcal{X}_2 =& \partial_t,\\
	\mathcal{X}_3 =& u\partial_u,\\
	\mathcal{X}_4 =& x\partial_x+2t\partial_t,\\
	\mathcal{X}_5 =& 2t\partial_x-xu\partial_u,\\
	\mathcal{X}_6 =& 4xt\partial_x+4t^2\partial_t-(x^2+2t)u\partial_u,\\
	\mathcal{X}_f =& f(x,t)\partial_u,	
\end{split}
\end{equation}
where $f(x,t)$ is any solution of \eqref{eq:Heat}. It is worth mentioning here that the infinite subalgebra --- an ideal actually --- $\mathcal{X}_f$ merely reflects the linearity of \eqref{eq:Heat}. This is a first hint that symmetries can capture important structural properties of differential equations. 

Finding the Lie point symmetries of a differential equation is just the first step in our journey through the astounding landscape illuminated by the ideia of symmetry. In the two subsections that follow we explore one generalization and one specialization  of Lie point symmetries relevant to our work. Next, we see a practical application of the  Lie point symmetries of the heat equation \eqref{eq:Heat}: constructing its fundamental solution, a key element of our study.

\subsubsection{The fundamental solution of \eqref{eq:Heat}}

We define as the \emph{fundamental solution} of the heat equation \eqref{eq:Heat} the solution of the particular \emph{Cauchy problem}
\begin{equation}\label{eq:FundamentalProblem}
	\begin{cases}
	\upsilon_t = \upsilon_{xx},\ -\infty<x<\infty,\, t>0,\\
	\upsilon(x,0) = \delta(x), \ -\infty<x<\infty.
	\end{cases}	
\end{equation}
Its knowledge permit us to express the solution of any other Cauchy problem for the heat equation with inicial data $\phi(x)$ as the convolution between the initial data and the fundamental solution, that is the function
$$
	u(x,t) = \upsilon(x,t)*\phi(x)
$$
satisfies the Cauchy problem
 $$
	\begin{cases}
	u_t = u_{xx},\ -\infty<x<\infty,\, t>0,\\
	u(x,0) = \phi(x), \ -\infty<x<\infty.
	\end{cases}
$$
Having the Lie algebra of point symmetries of equation~\eqref{eq:Heat} allows us to obtain its fundamental solution: all we need is to find the Lie subalgebra that keeps invariant the problema \eqref{eq:FundamentalProblem}. It is spanned by the vectors
\begin{align*}
	\mathcal{X}_5, && \mathcal{X}_4 - \mathcal{X}_3, && \mathcal{X}_6,
\end{align*}
see also \cite[p. $313$]{Ibr2009b}. The \emph{invariant solution} of these three vectors
\begin{align*}
	\mathcal{X}_5\left[u-\upsilon(x,t)\right]_{u=\upsilon(x,t)}=0,\\
	\left(\mathcal{X}_4 - \mathcal{X}_3\right)\left[u-\upsilon(x,t)\right]_{u=\upsilon(x,t)}=0,\\
	\mathcal{X}_6\left[u-\upsilon(x,t)\right]_{u=\upsilon(x,t)}=0,
\end{align*}
is the  fundamental solution that we are seeking --- none other than the well-known Gaussian function
$$
\upsilon(x,t) = \frac{e^{-\dfrac{x^2}{4 t}}}{2\sqrt{\pi t}}.
$$

\subsection{Generalized symmetries}\label{sec:GeneralizedSymmetries}

One can try to extend the algorithmic procedure described above to transformations of the form
\begin{align*}
	\bar x =& \bar x(x,t,u,u_x,u_t;\epsilon),\\
	\bar t =& \bar t(x,t,u,u_x,u_t;\epsilon),\\
	\bar u =& \bar u(x,t,u,u_x,u_t;\epsilon),\\
	\bar u_{\bar x} =& \phi(x,t,u,u_x,u_t;\epsilon),\\
	\bar u_{\bar t} =& \psi(x,t,u,u_x,u_t;\epsilon).
\end{align*}
That is, we now want to dictate the way that the first partial derivatives of $u$ change as well. Its representation as a infinitesimal generator is now
\begin{align*}
	\mathcal{X} =& \xi^1(x,t,u,u_x,u_t)\frac{\partial }{\partial x}+\xi^2(x,t,u,u_x,u_t)\frac{\partial }{\partial t} + \eta(x,t,u,u_x,u_t) \frac{\partial }{\partial u}+\eta_x(x,t,u,u_x,u_t) \frac{\partial }{\partial u_x}+\eta_t(x,t,u,u_x,u_t) \frac{\partial }{\partial u_t}. 
\end{align*}
Observe that we can determine the coefficients $\eta_x$ and $\eta_t$ in two different ways: either by using the functions $\phi$ and $\psi$ or by the prolongation formula that determine $\eta^{(1,0)}$ and $\eta^{(0,1)}$:
\begin{align*}
	\eta^{(1,0)}(x,t,u,u_x,u_t) =& D_x\left(\eta(x,t,u,u_x,u_t)-\xi^1(x,t,u,u_x,u_t)u_x\right)-D_x\left(\xi^2(x,t,u,u_x,u_t)u_t \right) +\xi^1(x,t,u,u_x,u_t)u_{xx}\\
	&+\xi^2(x,t,u,u_x,u_t)u_{xt},\\ 
	\eta^{(0,1)}(x,t,u,u_x,u_t) =& D_t\left(\eta(x,t,u,u_x,u_t)-\xi^1(x,t,u,u_x,u_t)u_x\right)-D_x\left(\xi^2(x,t,u,u_x,u_t)u_t \right) +\xi^1(x,t,u,u_x,u_t)u_{tx}\\
	&+\xi^2(x,t,u,u_x,u_t)u_{tt}.
\end{align*}
However, at this point an inconsistency arise: from the prolongation formula the two coefficients depend on the second order derivatives while the functions $\phi$ and $\phi$ do not! That means that in general no symmetries of this type exists. Only if we impose the adicional conditions, 
\begin{align*}
	\xi^1=-\frac{\partial}{\partial u_x}\left(\eta(x,t,u,u_x,u_t)-\xi^1(x,t,u,u_x,u_t)u_x-\xi^2(x,t,u,u_x,u_t)u_t \right),\\
	\xi^2=-\frac{\partial}{\partial u_t}\left(\eta(x,t,u,u_x,u_t)-\xi^1(x,t,u,u_x,u_t)u_x-\xi^2(x,t,u,u_x,u_t)u_t \right),
\end{align*}
the so called \emph{contact conditions}, there might be a chance to find a symmetry of this kind. When it exists, we call it a \emph{contact symmetry}. 

Due to the necessity of a contact symmetry to satisfy a number of contact conditions its existence is very fragile: systems of differential equations do not admit any and the concept cannot be extended to involve derivatives of higher order. In order to really generalize the idea of a Lie point symmetry to involve higher order derivatives we need to overcome the fact that the coefficients of the prolonged infinitesimal generator will always contain derivatives of even higher orders not considered in the group transformation. The way to redeem this issue is by making a ``small'' leap forward and assume a group transformation that involves the derivatives of \emph{any} order: 
\begin{align*}
	\bar x =& \bar x(x,t,u,u_x,u_t,u_{xx},u_{xt},u_{tt},\dots;\epsilon),\\
	\bar t =& \bar t(x,t,u,u_x,u_t,u_{xx},u_{xt},u_{tt},\dots;\epsilon),\\
	\bar u =& \bar u(x,t,u,u_x,u_t,u_{xx},u_{xt},u_{tt},\dots;\epsilon),\\
	\bar u_{\bar x} =& \phi(x,t,u,u_x,u_t,u_{xx},u_{xt},u_{tt},\dots;\epsilon),\\
	\bar u_{\bar t} =& \psi(x,t,u,u_x,u_t,u_{xx},u_{xt},u_{tt},\dots;\epsilon),\\
	\bar u_{\bar x\bar x} =& \rho(x,t,u,u_x,u_t,u_{xx},u_{xt},u_{tt},\dots;\epsilon),\\
	\bar u_{\bar x\bar t} =& \sigma(x,t,u,u_x,u_t,u_{xx},u_{xt},u_{tt},\dots;\epsilon),\\
	\bar u_{\bar t\bar t} =& \tau(x,t,u,u_x,u_t,u_{xx},u_{xt},u_{tt},\dots;\epsilon),\\
	\vdots&
\end{align*}
We call this kind of symmetry \emph{generalized}, \emph{dynamic} or \emph{Lie-Bäcklund} symmetry. Its representation as an infinitesimal vector takes the form
\begin{align*}
	\mathcal{X} =& \xi^1(x,t,u,u_x,u_t,u_{xx},u_{xt},u_{tt},\dots)\frac{\partial }{\partial x}+\xi^2(x,t,u,u_x,u_t,u_{xx},u_{xt},u_{tt},\dots)\frac{\partial }{\partial t}+ \eta(x,t,u,u_x,u_t,u_{xx},u_{xt},u_{tt},\dots) \frac{\partial }{\partial u}+\cdots, 
\end{align*}
where we determine the coefficients for the derivatives of you $u$ by applying a suitable version of the prolongation fórmula. It is easy to verify that the generalized symmetries
\begin{align*}
	&\xi^1(x,t,u,u_x,u_t,u_{xx},u_{xt},u_{tt},\dots)D_x =\xi^1(x,t,u,u_x,u_t,u_{xx},u_{xt},u_{tt},\dots)\left(\frac{\partial}{\partial x} + u_x\frac{\partial}{\partial u}+u_{xx}\frac{\partial}{\partial u_x}+ u_{xt}\frac{\partial}{\partial u_t}+\cdots\right)
  \intertext{and}
	&\xi^2(x,t,u,u_x,u_t,u_{xx},u_{xt},u_{tt},\dots)D_t =\xi^2(x,t,u,u_x,u_t,u_{xx},u_{xt},u_{tt},\dots)\left(\frac{\partial}{\partial t} + u_t\frac{\partial}{\partial u}+u_{xt}\frac{\partial}{\partial u_x}+ u_{tt}\frac{\partial}{\partial u_t}+\cdots\right),
\end{align*}
are symmetries admitted by \emph{any} differential equation and hence, in this context, we can consider them as \emph{trivial} symmetries in the same fashion that the zero vector is a trivial symmetry  in the context of Lie point symmetries. The Lie algebra generated by all these trivial symmetries is an ideal of the Lie algebra of all the generalized symmetries of a given differential equation, so by constructing the quotient space by this ideal we obtain the Lie algebra of all the non trivial generalized symmetries. An element of such a quotient space has the form
\begin{align*}
	\mathcal{X} =&\left(\eta(x,t,u,u_x,u_t,u_{xx},u_{xt},u_{tt},\dots)-\xi^1(x,t,u,u_x,u_t,u_{xx},u_{xt},u_{tt},\dots)u_x- \xi^2(x,t,u,u_x,u_t,u_{xx},u_{xt},u_{tt},\dots)u_t\right) \frac{\partial }{\partial u}+\cdots\\
	=& Q(x,t,u,u_x,u_t,u_{xx},u_{xt},u_{tt},\dots)\frac{\partial }{\partial u}+D_xQ\frac{\partial }{\partial u_x}+D_tQ\frac{\partial }{\partial u_t}+D_x^2 Q\frac{\partial }{\partial u_{xx}}+D_xD_tQ\frac{\partial }{\partial u_{xt}}+D_t^2 Q\frac{\partial }{\partial u_{tt}}+\cdots, 
\end{align*}
this is the \emph{characteristic form} of a generalized symmetry. Observe how much simpler is now to determine the coefficients for the derivatives of $u$!

To obtain a generalized symmetry we first fix the form of the the characteristic function $Q$, for instance $Q = u_{xxx}- Q(x,t,u,u_x,u_{xx})$, and then we employ the linearized symmetry condition \eqref{eq:LinSymCond} to obtain the determining equations that will specify $Q$.

\subsection{Approximate symmetries}\label{sec:ApproximateSymmetries}

When a differential equation involves a small parameter $\varepsilon\ll1$ --- like in the case of equation~\eqref{eq:PerturbedEq} --- we deal with a \emph{perturbed differential equation}. In our case, we consider \eqref{eq:PerturbedEq} as a perturbation of the heat~equation~\eqref{eq:Heat}. On top of the types of symmetries described so far there is another one type of symmetry that we can seek out for this kind of differential equations: we can assume that a transformation, or its corresponding infinitesimal generator, is a symmetry up to order $\varepsilon$, that is
\begin{equation}\label{eq:ApproxLinSymCond}
	\mathcal{X}^{(2)}[u_t-u_{xx}+\varepsilon u_{xxxx}]\Big\rvert_{u_t=u_{xx}-\varepsilon u_{xxxx}} = {\scriptstyle\mathcal{O}}(\varepsilon), 
\end{equation}
where the vector $\mathcal X$ has now the form
$$
	\mathcal{X} = \mathcal{X}_0+\varepsilon \mathcal{X}_\varepsilon.
$$
We call this type of symmetry an \emph{approximate symmetry}, \cite{Ibr2008a}. The vector $\mathcal{X}_0$  can be any Lie point symmetry of the unperturbed differential equation --- in our case the heat equation --- including the null vector. 

If for a symmetry of \eqref{eq:Heat}, $\mathcal Y$, there is an approximate symmetry $\mathcal Y+\varepsilon\mathcal{Y}_\varepsilon$ of \eqref{eq:PerturbedEq} we say that $\mathcal Y$ is \emph{stable}, that is, it is not lost due to the perturbation of the initial differential equation. Furthermore if this is valid for every element of the Lie algebra $\mathcal L$ we say that $\mathcal L$ is \emph{stable}, \cite[p. $27$]{Ibr2008a}. Stability of the Lie algebra means that the Lie algebra of approximate symmetries of equation~\eqref{eq:PerturbedEq} can be viewed as  a \emph{deformation} of the Lie algebra of point symmetries of equation~\eqref{eq:Heat}, and as such equation~\eqref{eq:PerturbedEq} \emph{inherits} all the results that Lie point symmetries yield for the heat equation as their perturbations. Moreover, it guarantees the existence of a transformation that connects the two differential equations.  Stabilizing $\mathcal L$ will be a crucial step in expressing the fundamental solution of equation~\eqref{eq:PerturbedEq} as a (singular) perturbation of the fundamental solution of the heat equation, $\upsilon(x,t)$.

An important consequence of this variation of the algorithmic process described in subsection~\ref{sec:LiPointSymmetries} is that now the Lie algebra of Lie point symmetries of the heat equation~\eqref{eq:Heat} appears as a deformation of the null vector, $\varepsilon\,\kappa_i \mathcal{X}_i$. In the literature these approximate symmetries are considered  trivial included only in order to guarantee the closeness of the Lie algebra of approximates symmetries, see also \cite[p. $27$]{Ibr2008a}. In the next section we shall see that this is far from the truth.

 The symmetries of a differential equation --- like the DNA for a living organism ---  provide essencial structural information that help us to identify and reveal important properties, reduce the order and even obtain solutions, as we have already seen, to name a few. Indeed, one can ponder why not just substitute $u(x,t)=u_0(x,t)+\varepsilon u_1(x,t)$ and solve the resulting system of PDEs of the second order. By using approximate symmetries we don't just make a substitution but we obtain the suitable invariant form of the solution that after a substitution will yield a system of ODEs of the second order hence facilitating  the construction of the approximate solution.

Evidently, getting the Lie algebra from \eqref{eq:LinSymCond} or \eqref{eq:ApproxLinSymCond} on the one hand involves a great deal of copious and error-prone calculations and on the other hand is a completely algorithmic procedure. These facts render the use of \emph{computer algebra systems} --- another facet of enhanced modern group analysis --- essential. For our purposes we employed the symbolic package SYM for Wolfram Mathematica\texttrademark, \cite{DiTs2k5a, MathematicaV14}. With the content of this section we barely skimmed the surface of this remarkably beautiful theory, for more details see also \cite{Ibr2008a,  Olver2k,Ovsi82, BluKu89, Ibra85a, Ste90, Hydon2k, Ibra95c}.

\section{Main results}

\subsection{The fundamental approximate solution of \eqref{eq:PerturbedEq}}

We start our study with the approximate Lie point symmetries of equation~\eqref{eq:PerturbedEq}.
\begin{theorem}
	The finite component of the Lie algebra of approximate Lie point symmetries of the perturbed PDE
	$$
		u_t=u_{xx} - \varepsilon u_{xxxx}
	$$
	is spanned by the vectors
	\begin{align*}
		\mathcal{Y}_1 =& \mathcal{X}_1,& \mathcal{Y}_5 =& \varepsilon\mathcal{X}_1\\
		\mathcal{Y}_2 =& \mathcal{X}_2,& \mathcal{Y}_6 =& \varepsilon\mathcal{X}_2\\
		\mathcal{Y}_3 =& \mathcal{X}_3,& \mathcal{Y}_7 =& \varepsilon\mathcal{X}_3\\
		--&- &\mathcal{Y}_8 =& \varepsilon\mathcal{X}_4, \\
		\mathcal{Y}_4 =& \mathcal{X}_5+\varepsilon\left(2x\partial_t\right),& \mathcal{Y}_{9} =& \varepsilon\mathcal{X}_5\\
		--&- &\mathcal{Y}_{10} =& \varepsilon\mathcal{X}_6.
	\end{align*}
	where $\mathcal{X}_i$ are given in \eqref{Heat:LieAlgebra}.
\end{theorem}
\begin{proof}
Direct use of the Lie algorithm up to order $\varepsilon$ --- see also subsection~\ref{sec:ApproximateSymmetries} --- using SYM,  \cite[p. $27$]{Ibr2008a}.
\end{proof}

The Lie algebra is not stable. Due to the perturbation we loose the vectors $\mathcal{X}_4$ and $\mathcal{X}_6$, which are crucial for inheriting the fundamental solution. So, we need a way to stabilize these two vectors. To accomplish that we need to consider that the deformed part of the two vector is a contact or a  generalized symmetry, see subsection~\ref{sec:GeneralizedSymmetries}. That is, we shall look for approximates symmetries of the form
\begin{align*}
	\mathcal{X}_4+\varepsilon\, Q(x,t,u,u_x,u_{xx},u_{xxx},\dots)\partial_u && \mathcal{X}_6+\varepsilon\, R(x,t,u,u_x,u_{xx},u_{xxx},\dots)\partial_u.
\end{align*}
This idea was first insinuated in \cite[p. $35$, eq.~$(4.1.9)$]{Ibr2008a}  and then fully realized in \cite{TarChe2021a}. So, by applying the Lie algorithm to this mixed type of symmetry we arrive to the following result
\begin{theorem}
	By allowing the perturbed part of the approximate symmetries to be a Lie-Bäcklund symmetry, we obtain the following approximate symmetries of~\eqref{eq:PerturbedEq}:
	\begin{align*}
		\mathcal{Z}_{1} =& \mathcal{X}_4 + \varepsilon\left(-xu_{xxx}\partial_u\right),\\
		\mathcal{Z}_{2} =& \mathcal{X}_6 + \varepsilon\left(3x\partial_x-x\left(3xu_{xx}+4tu_{xxx}\right)\partial_u\right).
	\end{align*}
\end{theorem} 
\begin{proof}
Direct use of the Lie algorithm up to order $\varepsilon$ for a mixed approximate symmetry of the form
$$
	\kappa \mathcal{X}_4+\lambda\mathcal{X}_6 + \varepsilon \,Q(x,t,u,u_x,u_{xx},u_{xxx})\partial_u
$$
using SYM.
\end{proof}

With the whole Lie algebra of Lie point symmetries of \eqref{eq:Heat} stabilized we can now obtain the approximate fundamental solution of equation~\eqref{eq:PerturbedEq} as a perturbation of the fundamental solution of equation~\eqref{eq:Heat}.

\begin{theorem}
The approximate fundamental solution of equation~\eqref{eq:PerturbedEq} is 
\begin{equation}\label{eq:ApproximateFundamentalSolution1}
		\upsilon_\varepsilon(x,t) = \frac{e^{-\dfrac{x^2}{4 t}}}{2\sqrt{\pi t}} - \varepsilon \frac{12 t^2-12t x^2+x^4}{32\sqrt{\pi t^7} }e^{-\dfrac{x^2}{4 t}}.
\end{equation}
\end{theorem}
\begin{proof}
	We need to find the Lie subalgebra of approximate symmetries that keeps invariant the Cauchy problem 
	\begin{equation}\label{eq:ApproximateFundamentalProblem}
	\begin{cases}
	\upsilon_{\varepsilon,\,t} = \upsilon_{\varepsilon,\,xx}-\varepsilon \upsilon_{\varepsilon,\,xxxx},\ -\infty<x<\infty,\, t>0,\\
	\upsilon_\varepsilon(x,0) = \delta(x)+\varepsilon\cdot0, \ -\infty<x<\infty,
	\end{cases}	
	\end{equation}
up to order $\varepsilon$.
We know that the Lie subalgebra spanned by the vectors
\begin{align*}
	\mathcal{X}_5, && \mathcal{X}_4 - \mathcal{X}_3, && \mathcal{X}_6,
\end{align*}
keeps invariant the Cauchy problem \eqref{eq:FundamentalProblem}, so we can start our analysis from its deformation spanned by 
\begin{align*}
	\mathcal{Y}_4, && \mathcal{Z}_1 - \mathcal{X}_3, && \mathcal{Z}_2.
\end{align*}
Next, we need to find the Lie subalgebra that satisfies the initial condition, that is, we need all the vectors that keep invariant the border $t=0$, for $t=0$ the symmetries should keep the pole of the Dirac function $x=0$ invariant, and finally they should keep invariant the condition itself $\upsilon_\varepsilon(x,0) = \delta(x)+\varepsilon\cdot0$. By mere inspection of the deformation of the the linear combination of vectors $\kappa \mathcal{Y}_4+ \lambda\left(\mathcal{Z}_1 - \mathcal{X}_3\right)+\mu\mathcal{Z}_2$,
$$
3\mu x\partial_x+2\kappa x\partial_t-\left(\lambda  u_{xxx}+\mu \left(3xu_{xx}+4tu_{xxx}\right) \right)x\partial_u,
$$
we can see that it leaves invariant the first condition if, and only if, $\kappa=0$, while the second condition is satisfied for any values of $\lambda$ and $\mu$. As for the last condition, by taking into account that
\begin{align*}
  	&\left(\xi^1(x,t,u,\dots)\frac{\partial }{\partial x}+\xi^2(x,t,u,\dots)\frac{\partial }{\partial t} + \eta(x,t,u,\dots) \frac{\partial }{\partial u}\right)[\delta(x-x_0)] =- \left. \frac{\partial}{\partial x}\xi_1(x,t,u,\dots) \right\rvert_{x=x_0}\delta(x-x_0),
\end{align*}
see also \cite[p. $304$]{Ibr2009b}, we arrive at the restriction 
\begin{align*}
  -\left(\lambda x \delta^{\prime\prime\prime}(x)+3\mu x^2\delta^{\prime\prime}(x) \right) = - 3\mu\delta(x)\implies& 6\mu\delta(x)-3\lambda\delta^{\prime\prime}(x) = 3\mu\delta(x) \implies \mu\delta(x)-\lambda\delta^{\prime\prime}(x) = 0.
\end{align*}
This condition is satisfied only when $\lambda=\mu=0$. Something is amiss: we disconsider \emph{a priori} the deformations of the null vector, which by definition keeps invariant the problem at hand! 

Since the final condition involves the Dirac delta function and its second derivative, we seek null deformations that reproduce exactly these two terms under the same conditions. The simplest such deformations are
\begin{align*}
	\mathcal{Y}_7=\varepsilon\,\mathcal{X}_3=\varepsilon\, u\partial_u &&\text{and}&& \mathcal{Z}_3=\varepsilon\,u_{xx}\partial_u.
\end{align*}
Indeed, verifying the last condition for the linear combination of vectors $\lambda\left(\mathcal{Z}_1 - \mathcal{Y}_y\right)+\mu\mathcal{Z}_2+\alpha\mathcal{X}_3+\beta\mathcal{Z}_3$ we arrive at the relation
$$
	6\mu\delta(x)-3\lambda\delta^{\prime\prime}-\alpha\delta(x)-\beta\delta^{\prime\prime}(x) = 3\mu\delta(x)\implies (3\mu-\alpha)\delta(x)-(\beta+3\lambda)\delta^{\prime\prime}=0.
$$

Therefore, the Lie subalgebra spanned by the vectors 
\begin{align*}
	\mathcal{Z}_1 - \mathcal{X}_3-3\mathcal{Z}_3  && \mathcal{Z}_2+3\varepsilon\,\mathcal{X}_3
\end{align*}
keeps invariant the Cauchy~problem~\eqref{eq:ApproximateFundamentalProblem}. Since we are working in $1+1$ dimensions one of the two base elements is enough, we shall use the simplest one, $\mathcal{Z}_1 - \mathcal{X}_3-3\mathcal{Z}_3$. Thus, the function $\upsilon_\varepsilon$ is the invariant solution of that vector, that is,
$$
\left(\mathcal{Z}_1 - \mathcal{X}_3-3\mathcal{Z}_3\right)\left[u-\upsilon_\varepsilon(x,t)\right]_{u=\upsilon_\varepsilon(x,t)}=\mathcal{o}(\varepsilon)
$$
where
$$
	\upsilon_\varepsilon(x,t) = f(x,t)+\varepsilon g(x,t).
$$
By solving the resulting system we find that
\begin{align*}
	f(x,t) =& \frac{1}{\sqrt t}\phi\left(\frac{x^2}{t}\right)
  \intertext{and}
  g(x,t) =& \frac{1}{\sqrt{t}}\psi\left(\frac{x^2}{t}\right)+\frac{1}{t^{7/2}}\left(3t^2\phi^\prime+4x^2(3t\phi^{\prime\prime}+x^2\phi^{\prime\prime\prime})\right).
\end{align*}
Since $\upsilon_\varepsilon$ is the approximate solution of \eqref{eq:PerturbedEq} by substituting it into the Cauchy problem  we can determine the two functions $\phi$ and $\psi$, that is
\begin{align*}
	\phi(\zeta) = \frac{1}{2\sqrt{\pi}}e^{-\dfrac{\zeta}{4}} &&\text{and}&& \psi(\zeta) = 0.
\end{align*}
Thus, we arrive to the approximate fundamental solution
	$$
			\upsilon_\varepsilon(x,t) = \frac{e^{-\dfrac{x^2}{4 t}}}{2\sqrt{\pi t}} - \varepsilon \frac{12 t^2-12t x^2+x^4}{32\sqrt{\pi t^7}}e^{-\dfrac{x^2}{4 t}}.	
	$$
\end{proof}

So, when $\varepsilon\ll1$ the function 
$$
\frac{12 t^2-12t x^2+x^4}{32\sqrt{\pi t^7}}e^{-\dfrac{x^2}{4 t}}
$$
provides a very good approximation of the secondary diffusion that the model \eqref{eq:PerturbedEq} manifests. In figure~\ref{fig:ComparisonOfDiffusions1} we give a qualitative, and out of scale, comparison of the two diffusive processes.

\begin{figure}[H]
	\centering
	\includegraphics[scale=0.39]{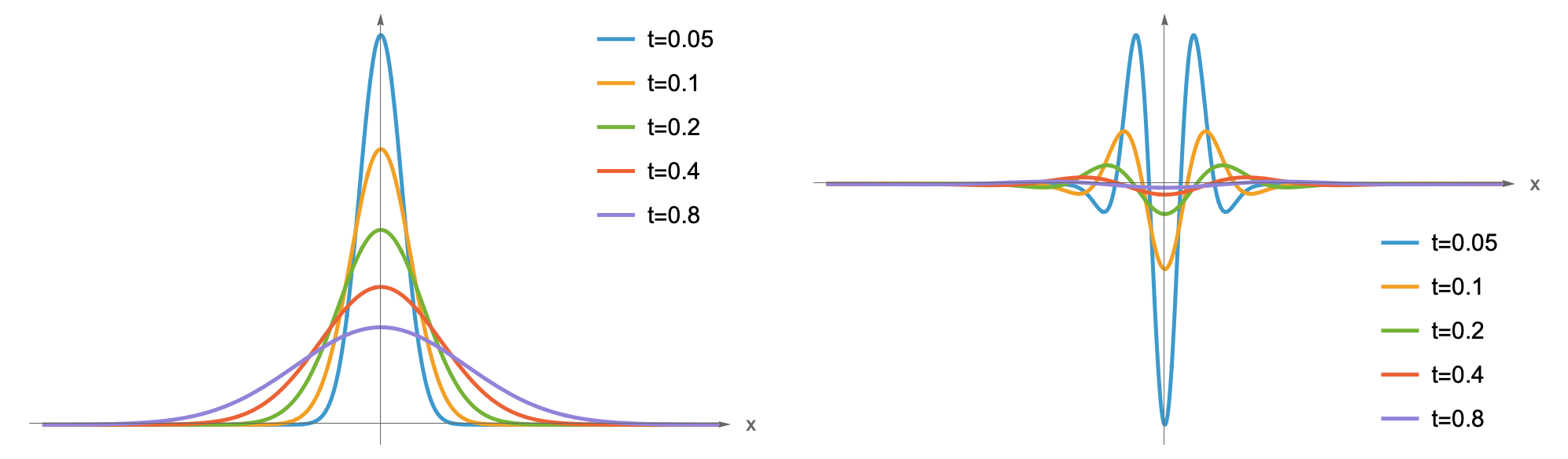}
	\caption{Qualitative comparison of the primary (left) and the secondary (right) diffusion for different values of $t$.}
	\label{fig:ComparisonOfDiffusions1}
\end{figure}

Equipped wih the approximate fundamental solution we can find the approximate solution of any Cauchy problem of the form
$$
\begin{cases}
	u_t = u_{xx}-\varepsilon u_{xxxx},\ -\infty<x<\infty,\, t>0,\\
	u(x,0) = f(x)+\varepsilon g(x), \ -\infty<x<\infty.
\end{cases}	
$$
The approximate solution is
\begin{equation}\label{eq:AppSolOfCauchyProblem1}
\begin{split}
	u(x,t) =& \upsilon_\varepsilon(x,t)*(f(x)+\varepsilon g(x)) = \int\limits_{-\infty}^\infty \left(\frac{e^{-\dfrac{\xi^2}{4 t}}}{2\sqrt{\pi t}} - \varepsilon \frac{12 t^2-12t \xi^2+\xi^4}{32\sqrt{\pi t^7}}e^{-\dfrac{\xi^2}{4 t}}\right)(f(x-\xi)+\varepsilon g(x-\xi))\,\mathrm{d}\xi\\
	=&\int\limits_{-\infty}^\infty \frac{e^{-\dfrac{\xi^2}{4 t}}}{2\sqrt{\pi t}} f(x-\xi)\,\mathrm{d}\xi + \varepsilon\int\limits_{-\infty}^\infty \left(\frac{e^{-\dfrac{x^2}{4 t}}}{2\sqrt{\pi t}}g(x-\xi) - \frac{12 t^2-12t \xi^2+\xi^4}{32\sqrt{\pi t^7}}e^{-\dfrac{\xi^2}{4 t}}f(x-\xi)\right)\,\mathrm{d}\xi.	
\end{split}
\end{equation} 
Next, we repeat the method just illustrated to equation~\eqref{eq:PerturbedEq2}.

\subsection{The fundamental approximate solution of \eqref{eq:PerturbedEq2}}

For equation~\eqref{eq:PerturbedEq2} the corresponding Cauchy problem that yields the fundamental solution takes the form
\begin{equation}\label{eq:FundamentalProblem2}
	\begin{cases}
	G_t = G_{xx}+\varepsilon\, x^2G_{xx} ,\ -\infty<x<\infty,\, t>0,\\
	G(x,0;y) = \delta(x-y), \ -\infty<x<\infty.
	\end{cases}	
\end{equation}
Instead of the problem \eqref{eq:FundamentalProblem2} we shall consider the equivalent \emph{family} of Cauchy problems:
 \begin{equation}\label{eq:FundamentalProblem2alt}
	\begin{cases}
	\upsilon_t = \upsilon_{xx}+\varepsilon (x+y)^2\upsilon_{xx} ,\ -\infty<x<\infty,\,-\infty<y<\infty,\, t>0,\\
	\upsilon(x,0) = \delta(x), \ -\infty<x<\infty.
	\end{cases}	
\end{equation}
For this family of PDEs we have the following theorem.
\begin{theorem}
	The finite component of the Lie algebra of approximate Lie point symmetries of the perturbed family of PDEs
	\begin{equation}\label{eq:PerturbedEq2alt}
		u_t=u_{xx} - \varepsilon\, (x+y)^2 u_{xxxx}
	\end{equation}
	is spanned by the vectors
	\begin{align*}
		--&- 									& \mathcal{Y}_5 =		& \varepsilon\mathcal{X}_1\\
		\mathcal{Y}_2 =& \mathcal{X}_2,						& \mathcal{Y}_6 =		& \varepsilon\mathcal{X}_2,\\
		\mathcal{Y}_3 =& \mathcal{X}_3,						& \mathcal{Y}_7 =		& \varepsilon\mathcal{X}_3,\\
		\mathcal{Y}_4 =& (x+y)\partial_x+2t\partial_t,	& \mathcal{Y}_8 =		& \varepsilon\mathcal{X}_4, \\
		--&- 									& \mathcal{Y}_{9} =	& \varepsilon\mathcal{X}_5,\\
		\mathcal{Y}_6 =& 4(x+y)t\partial_x+4t^2\partial_t -(2t+x(x+2y))u\partial_u	+\varepsilon x(x^3+4x^2y+6xy^2+4y^3)\partial_t,			& \mathcal{Y}_{10} =& \varepsilon\mathcal{X}_6.
	\end{align*}
	where $\mathcal{X}_i$ are given in \eqref{Heat:LieAlgebra}.
\end{theorem}
\begin{proof}
Direct use of the Lie algorithm up to order $\varepsilon$ using SYM.
\end{proof}

Once more, the Lie algebra is not stable. Due to the perturbation we loose the vectors $\mathcal{X}_1$ and $\mathcal{X}_5$, the latter, one of the three symmetries needed for inheriting the fundamental solution. So, we need a way to stabilize that vector. To accomplish that we need to consider that the deformed part of the two vectors is a contact or a  generalized symmetry, see also subsection~\ref{sec:GeneralizedSymmetries}. That is, we shall look for approximates symmetries of the form
\begin{align*}
	\mathcal{X}_1+\varepsilon\, Q(x,t,u,u_x,u_{xx},u_{xxx},\dots)\partial_u && \mathcal{X}_5+\varepsilon\, R(x,t,u,u_x,u_{xx},u_{xxx},\dots)\partial_u.
\end{align*}
By applying the Lie algorithm to this mixed type of symmetry we arrive to the following result:
\begin{theorem}
	By allowing the perturbed part of the approximate symmetries to be a Lie-Bäcklund symmetry, we obtain the following approximate symmetries of~\eqref{eq:PerturbedEq2alt}:
	\begin{align*}
		\mathcal{Z}_{1} =& \mathcal{X}_1 + \varepsilon\left(t+\frac{x^2}{2}+xy\right)u_{xxx}\partial_u,\\
		\mathcal{Z}_{2} =& \mathcal{X}_5 -\frac{\varepsilon}{2} \left[\left(6t(x+y)+x^2(x+3y)\right)u_{xx}+2t\left(2t-x^2-2xy-4y^2\right)u_{xxx}\right]\partial_u.
	\end{align*}
\end{theorem} 
\begin{proof}
Direct use of the Lie algorithm up to order $\varepsilon$ for a mixed approximate symmetry of the form
$$
	\kappa \mathcal{X}_1+\lambda\mathcal{X}_5 + \varepsilon \,Q(x,t,u,u_x,u_{xx},u_{xxx})\partial_u
$$
using SYM.
\end{proof}

With the whole Lie algebra of Lie point symmetries of \eqref{eq:Heat} stabilized we can now obtain the approximate fundamental solution of equation~\eqref{eq:PerturbedEq2alt} as a perturbation of the fundamental solution of equation~\eqref{eq:Heat}.

\begin{theorem}
The approximate fundamental solution of equation~\eqref{eq:PerturbedEq2alt} is 
\begin{align*}
  &\upsilon_\varepsilon(x,t) = \frac{e^{-\dfrac{x^2}{4 t}}}{2\sqrt{\pi t}} + \varepsilon \frac{84t^3+36t^2(3x-y)y-x^4(x^2+3xy+3y^2)+3tx^2(x^2+4xy+12y^2)}{96\sqrt{\pi t^7}}e^{-\dfrac{x^2}{4 t}},
\end{align*}

and as a consequence, the approximate fundamental solution of equation~\eqref{eq:PerturbedEq2} is
 \begin{equation}\label{eq:ApproximateFundamentalSolution2}
	\begin{split}
		&G_\varepsilon(x,t;y) = \frac{e^{-\dfrac{(x-y)^2}{4 t}}}{2\sqrt{\pi t}}+ \varepsilon \frac{84t^3+36t^2(3x-4y)y-(x-y)^4(x^2+xy+y^2)+3t(x-y)^2(x^2+2xy+9y^2)}{96\sqrt{\pi t^7}}e^{-\dfrac{(x-y)^2}{4 t}}
	\end{split}
\end{equation}
\end{theorem}
\begin{proof}
	We need to find the Lie subalgebra of approximate symmetries that keeps invariant the family of Cauchy problems~\eqref{eq:FundamentalProblem2alt} up to order $\varepsilon$.
We know that the Lie subalgebra spanned by the vectors
\begin{align*}
	\mathcal{X}_5, && \mathcal{X}_4 - \mathcal{X}_3, && \mathcal{X}_6,
\end{align*}
keeps invariant the Cauchy problem \eqref{eq:FundamentalProblem}, so we can start our analysis from its deformation spanned by 
\begin{align*}
	\mathcal{Z}_2, && \mathcal{Y}_4 - \mathcal{X}_3, && \mathcal{Y}_6.
\end{align*}
Next, we need to find the Lie subalgebra that satisfies the initial condition, that is, we need all the vectors that keep invariant the border $t=0$, for $t=0$ the symmetries should keep the pole of the Dirac function $x=0$ invariant, and finally they should keep invariant the condition itself $\upsilon_\varepsilon(x,0) = \delta(x)+\varepsilon\cdot0$. So, let the linear combination of vectors $\kappa \mathcal{Z}_2+ \lambda\left(\mathcal{Y}_4 - \mathcal{X}_3\right)+\mu\mathcal{Y}_6$, that is,
\begin{equation*}
\begin{split}
 	&(2\kappa t+(\lambda+\mu t)(x+y))\partial_x+2(\lambda+2\mu t)t\partial_t-(\kappa x+\lambda+\mu(2t+x(x+2y)))u\partial_u+\varepsilon\Bigl(\mu x(x^3+4x^2y+6xy^2+4y^3)\partial_t\\
  &\qquad-\frac{\kappa}{2}\left[\left(6t(x+y)+x^2(x+3y)\right)u_{xx}+2t\left(2t-x^2-2xy-4y^2\right)u_{xxx}\right]\partial_u\Bigr).
\end{split}
\end{equation*}
By simple inspection, we can see that it leaves invariant the first condition if, and only if, $\mu=0$, while the second condition is satisfied if, and only if, $\lambda=0$. As for the last condition, by taking into account that
\begin{align*}
  	\left(\xi^1(x,t,u,\dots)\frac{\partial }{\partial x}+\xi^2(x,t,u,\dots)\frac{\partial }{\partial t} + \eta(x,t,u,\dots) \frac{\partial }{\partial u}\right)[\delta(x-x_0)] =- \left. \frac{\partial}{\partial x}\xi_1(x,t,u,\dots) \right\rvert_{x=x_0}\delta(x-x_0),
\end{align*}
we arrive at the restriction 
\begin{align*}
  -\kappa x \delta(x) -\frac{\kappa}{2}\varepsilon x^2(x+3y)\delta^{\prime\prime}(x) = 0\overset{x^3\delta^{\prime\prime}(x)=0}{\implies}-\frac{3\kappa}{2}yx^2\delta^{\prime\prime}(x) =0\implies -3\kappa y\delta(x)=0.
\end{align*}
This condition is satisfied only when $\kappa=0$. As in the previous section, we shall use one of the null deformations  to rectify the problem. In this case the situation is quite obvious, we need to employ the approximate symmetry $\varepsilon u\partial_u$. Indeed, a straightforward verification shows that the vector $\mathcal{Z}_2 + 3y\mathcal{X}_3$ satisfies the initial condition!

Thus, the function $\upsilon_\varepsilon$ is the invariant solution of that vector, that is,
$$
\left(\mathcal{Z}_2 + 3y\mathcal{X}_3\right)\left[u-\upsilon_\varepsilon(x,t)\right]_{u=\upsilon_\varepsilon(x,t)}=\mathcal{o}(\varepsilon)
$$
where
$$
	\upsilon_\varepsilon(x,t) = f(x,t)+\varepsilon g(x,t).
$$
By solving the resulting system --- the same way we did before --- we arrive at the approximate fundamental solution that we are looking for:
\begin{align*}
  &\upsilon_\varepsilon(x,t) = \frac{e^{-\dfrac{x^2}{4 t}}}{2\sqrt{\pi t}}+ \varepsilon \frac{84t^3+36t^2(3x-y)y-x^4(x^2+3xy+3y^2)+3tx^2(x^2+4xy+12y^2)}{96\sqrt{\pi t^7}}e^{-\dfrac{x^2}{4 t}}.
\end{align*}
Finally, by using the mapping $x\mapsto x-y$ we obtain the approximate fundamental solution of \eqref{eq:PerturbedEq2}:
\begin{align*}
    &G_\varepsilon(x,t;y) = \frac{e^{-\dfrac{(x-y)^2}{4 t}}}{2\sqrt{\pi t}}+ \varepsilon \frac{84t^3+36t^2(3x-4y)y-(x-y)^4(x^2+xy+y^2)+3t(x-y)^2(x^2+2xy+9y^2)}{96\sqrt{\pi t^7}}e^{-\dfrac{(x-y)^2}{4 t}}.
\end{align*}
\end{proof}
 In figure~\ref{fig:ComparisonOfDiffusions2} we give a qualitative, and out of scale, comparison of the two diffusive processes for $y=0$. All other values of $y$ corresponding to translations of these graphs.  

\begin{figure}[H]
	\centering
	\includegraphics[scale=0.42]{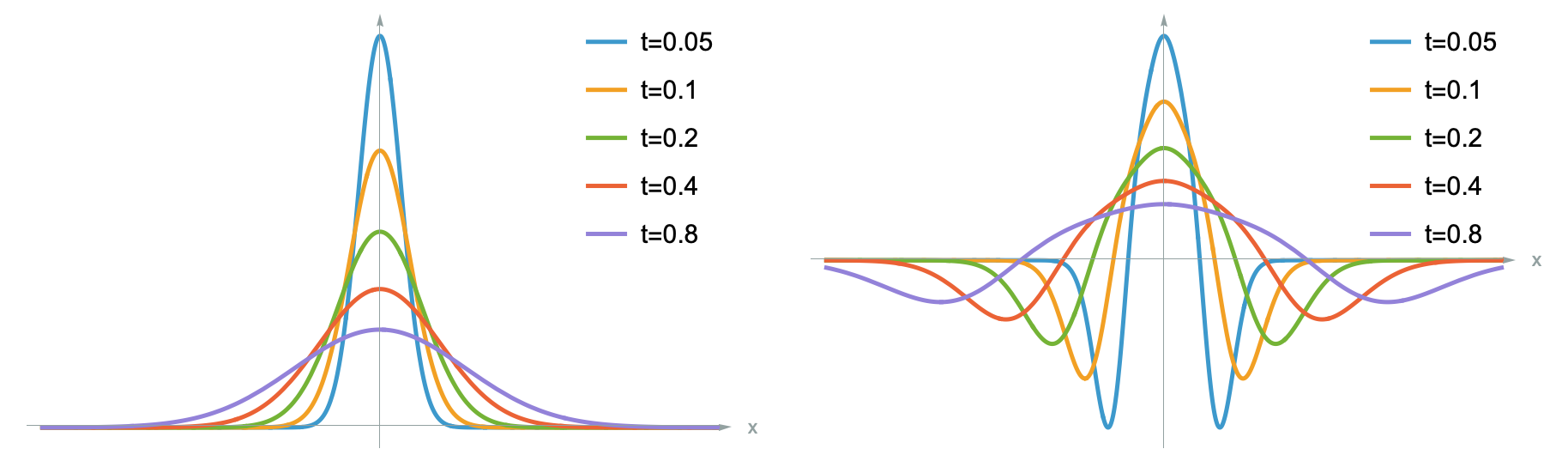}
	\caption{Qualitative comparison of the primary (left) and the secondary (right) diffusion for different values of $t$ and $y=0$.}
	\label{fig:ComparisonOfDiffusions2}
\end{figure}

As for the Cauchy problem of the form
$$
\begin{cases}
	u_t = u_{xx}-\varepsilon x^2 u_{xxxx},\ -\infty<x<\infty,\, t>0,\\
	u(x,0) = f(x)+\varepsilon g(x), \ -\infty<x<\infty,
\end{cases}	
$$
its approximate solution is
\begin{equation}\label{eq:AppSolOfCauchyProblem2}
\begin{split}
	&u(x,t) = G_\varepsilon(x,t;y)*(f(y)+\varepsilon g(y))= \int\limits_{-\infty}^\infty \frac{e^{-\dfrac{(x-\xi)^2}{4 t}}}{2\sqrt{\pi t}} f(\xi)\,\mathrm{d}\xi\\
	&+ \varepsilon\int\limits_{-\infty}^\infty \left(\frac{e^{-\dfrac{(x-\xi)^2}{4 t}}}{2\sqrt{\pi t}}g(\xi)+ \frac{84t^3+36t^2(3x-4\xi)\xi-(x-\xi)^4(x^2+\xi x+\xi^2)+3t(x-\xi)^2(x^2+2\xi x+9\xi^2)}{96\sqrt{\pi t^7}}e^{-\dfrac{(x-\xi)^2}{4 t}}f(\xi)\right)\,\mathrm{d}\xi.	
\end{split}
\end{equation}

\section{Application and numerical comparison}

Suppose that $f(x)=e^{-x^2}$ and $g(x)=0$. The two approximate solutions found, \eqref{eq:AppSolOfCauchyProblem1} and  \eqref{eq:AppSolOfCauchyProblem2}, take the particular forms  
$$
	u_1(x,t) = \frac{e^{-\dfrac{x^2}{1+4t}}}{\sqrt{1+4t}}-\varepsilon \frac{4e^{-\dfrac{x^2}{1+4t}}t\left(3+48t^2-12x^2+4x^4+24t(1-2x^2)\right)}{\sqrt{(1+4t)^9}}
$$
and 
\begin{align*}
  u_2(x,t) =& \frac{e^{-\dfrac{x^2}{1+4t}}}{\sqrt{1+4t}}\\
  &+\varepsilon \frac{4e^{-\dfrac{x^2}{1+4t}}t\left(12 (4 t+1) (t (4 t+3)+3) x^4+9 (16 t-1) (4 t x+x)^2+3 t (4 t+1)^3 (28 t-3)-4 (4 t (4 t+3)+3) x^6\right)}{3\sqrt{(1+4t)^{13}}},
\end{align*}
respectively. In figures~\ref{fig:ComparisonOfDiffusions3} and \ref{fig:ComparisonOfDiffusions4} we give a qualitative, and out of scale, comparison of the evolution in time of the two diffusive processes for the two particular solutions just given. In figure~\ref{fig:RetentionEffect} we visualize the retention effect of the secondary flux for both solutions.

\begin{figure}[H]
	\centering
	\includegraphics[scale=0.39]{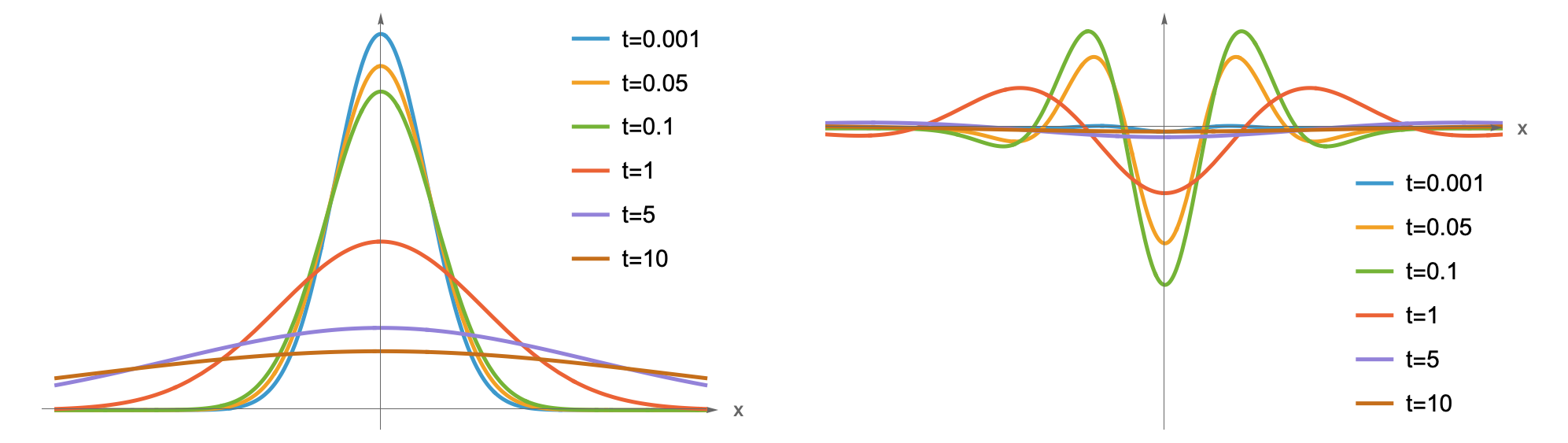}
	\caption{Qualitative comparison of the evolution in time of the primary (left) and secondary (right) diffusion processes.}
	\label{fig:ComparisonOfDiffusions3}
\end{figure}

\begin{figure}[H]
	\centering
	\includegraphics[scale=0.25]{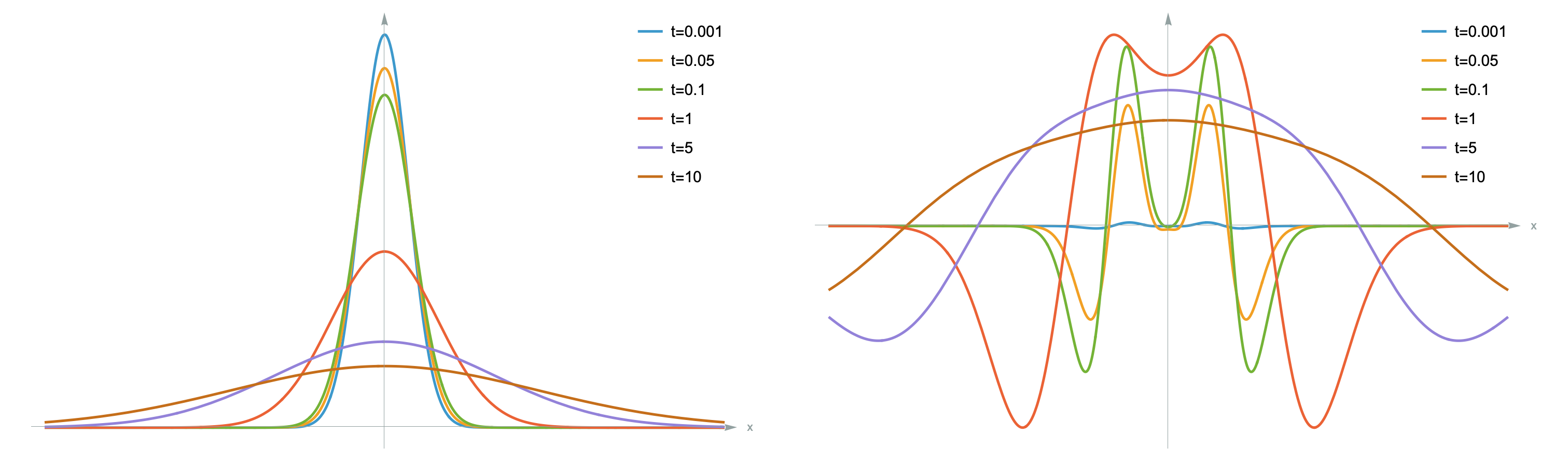}
	\caption{Qualitative comparison of the evolution in time of the primary (left) and secondary (right) diffusion processes.}
	\label{fig:ComparisonOfDiffusions4}
\end{figure}

\begin{figure}[H]
	\centering
	\includegraphics[scale=.4]{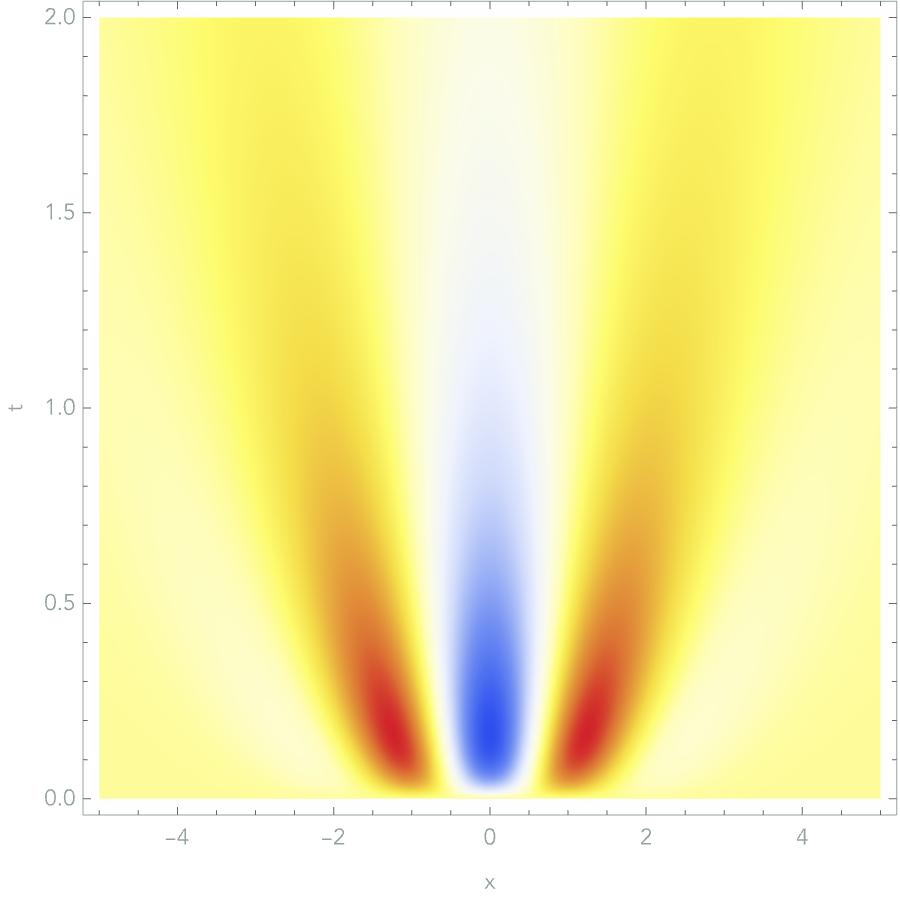}\hfill\includegraphics[scale=.4]{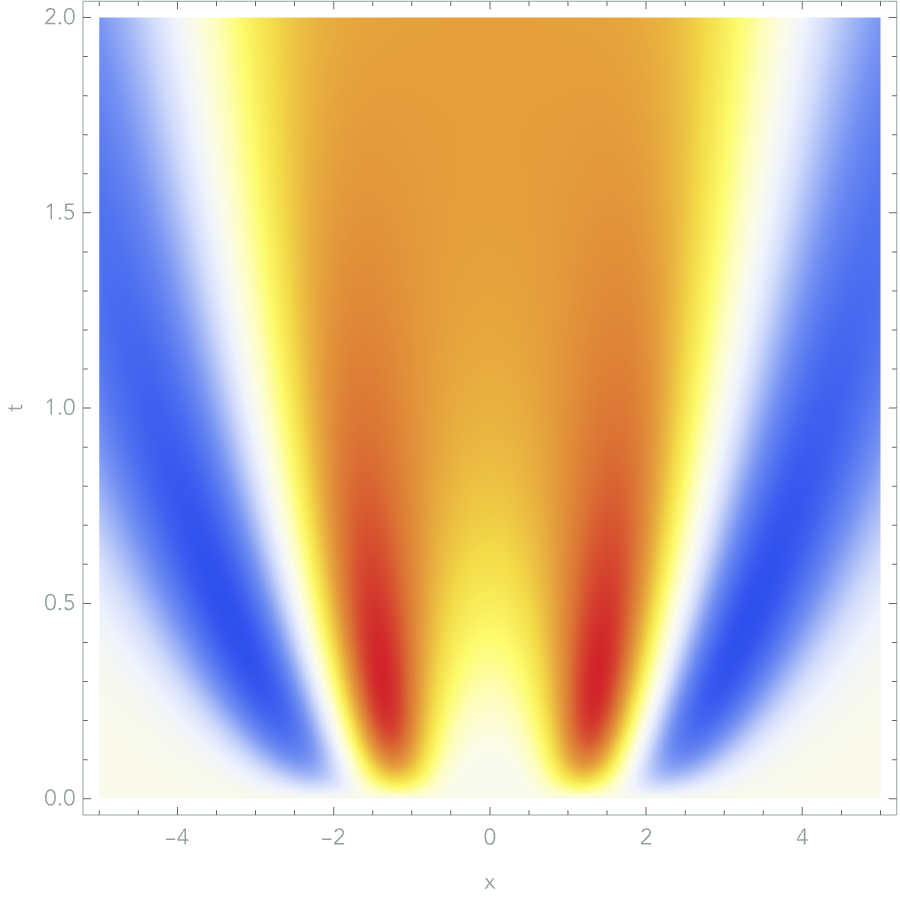}
	\caption{The retention effect of the secondary flux on $u_1$ (left) and $u_2$ (right). Red corresponds to higher retention and blue to lower one.}
	\label{fig:RetentionEffect}
\end{figure}

In figures ~\ref{fig:NumericalComparison1} and \ref{fig:NumericalComparison2} we give an error estimate of each one of the two approximate solutions found against their numerical counterpart using the $L^2$ norm,
$
	\displaystyle \lVert\cdot \rVert_{L^2} = \sqrt{\int\limits_{-5}^{5}\lvert\cdot\rvert^2\,\mathrm{d}x},
$
for different values of time. For obtaining each one of the numerical solutions we employed the Mathematica's function \href{https://reference.wolfram.com/language/ref/ParametricNDSolve.html}{ParametricNDSolve} with periodic boundary conditions 
\begin{align*}
	 u(-10,t)=u(10,t),&&  u_x(-10,t)=u_x(10,t)
\end{align*}
and options \emph{Method -> \{"MethodOfLines", "SpatialDiscretization" -> \{"TensorProductGrid", "DifferenceOrder" -> "Pseudospectral","MaxStepSize" -> 0.01\}}. From these figures we can deduce that the error remains most of the time in a range of order $\varepsilon^2$. 

\begin{figure}[H]
	\centering
	\includegraphics[scale=0.44]{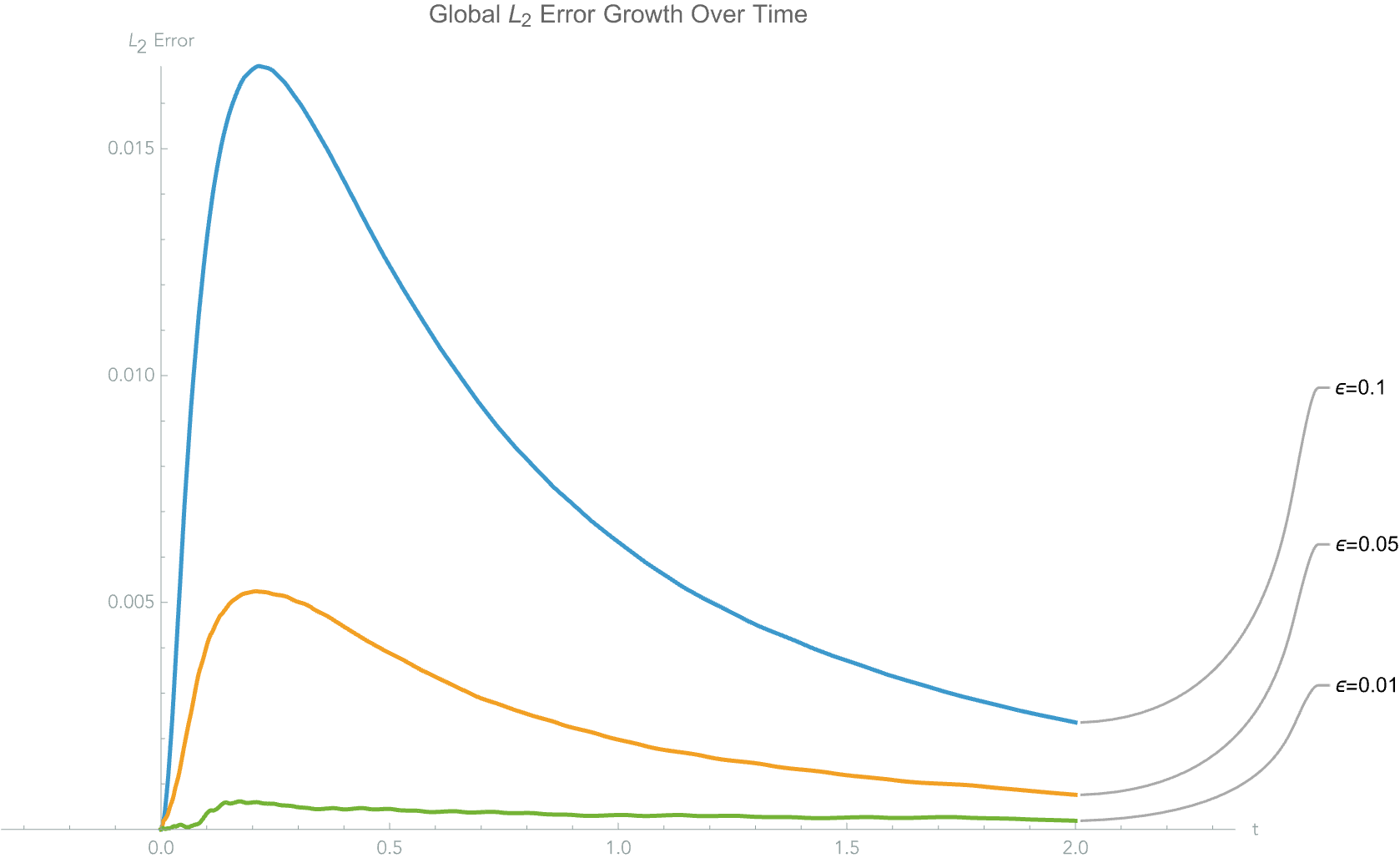}
	\caption{Error growth for various values of $\varepsilon$ and $0\le t\le 2$.}
	\label{fig:NumericalComparison1}
\end{figure}

\begin{figure}[H]
	\centering
	\includegraphics[scale=0.44]{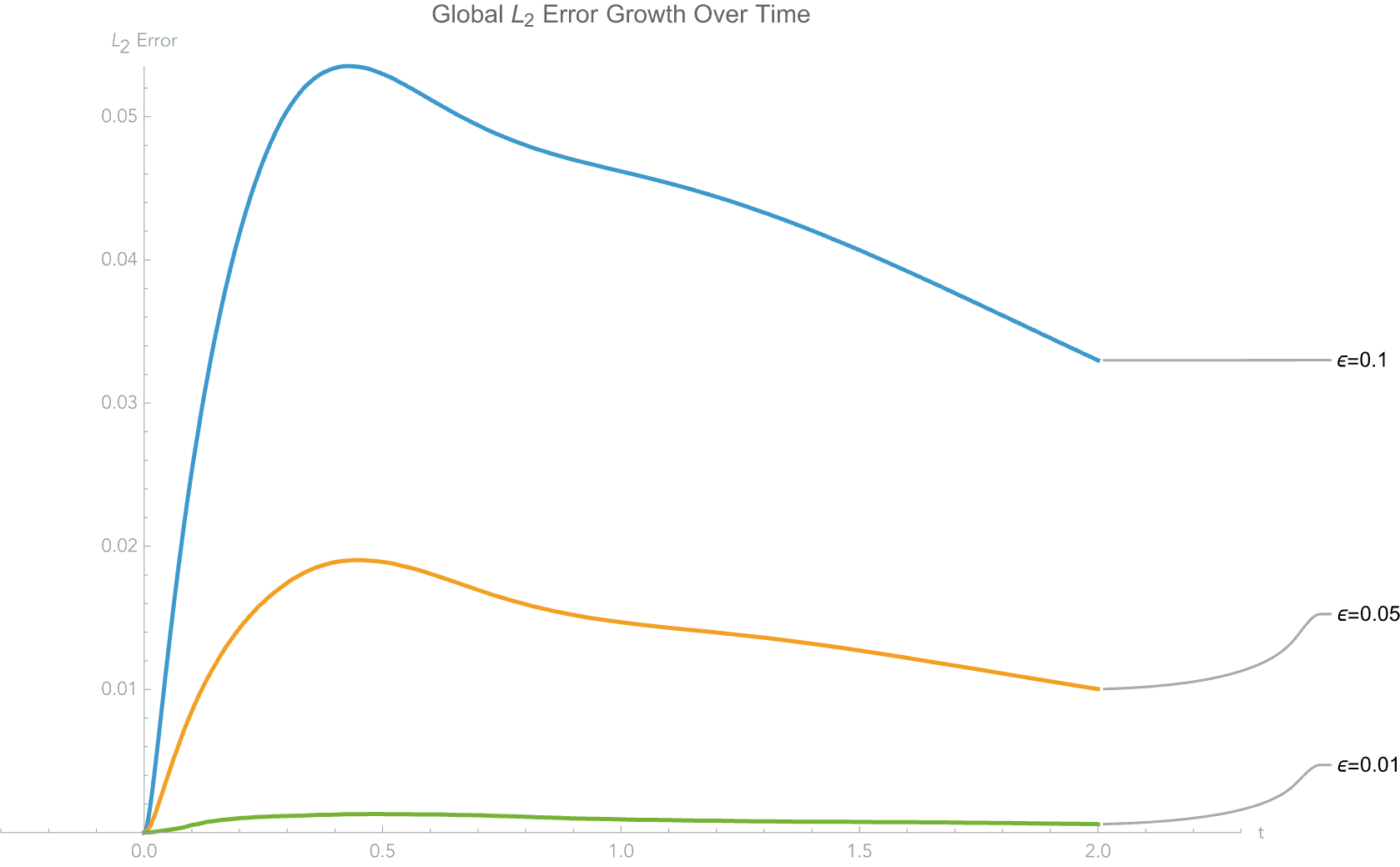}
	\caption{Error growth for various values of $\varepsilon$ and $0\le t\le 2$.}
	\label{fig:NumericalComparison2}
\end{figure}

\section{Concluding remarks}

In this work we started with  the governing equation for diffusion with convection,
$$
	p_t(\mathbf x,t) = \nabla\cdot\left(\beta\mathbf D\cdot\nabla p(\mathbf x,t)\right)-\nabla\cdot\left(\beta(1-\beta )\mathbf R\nabla\cdot\left(\nabla\left(\mathbf C\cdot\nabla p(\mathbf x,t)\right)\right)\right),
$$
and reduced it to its  bare essentials: an equation modelling only two fluxes --- Fickian diffusion and  retention-induced diffusion,
$$
u_t=u_{xx} - rc(x)\,u_{xxxx}.
$$
Our goal was to treat this as a (singular) perturbation of the heat equation, separating the two diffusion processes and expressing them in closed form. To do so, we applied advanced tools from modern group analysis, that is, approximate symmetries and generalized (Lie-Bäcklund) symmetries, instead of adopting a more traditional approach, for instance Fourier transforms. 

Symmetries as general-purpose tools they may initially underperform compared to specialized methods ---  especially when restricted to the simplest, and most known, kind of symmetries, the Lie point ones. This explains the common view in the literature that  symmetry-based  solutions apply only to limiting cases, where dynamical systems evolve toward  more symmetric states. Our main motivation was to provide a blueprint showing how more sophisticated kinds of symmetries can address problems that might seem out of reach for Sophus Lie's brilliant idea. Indeed, this became apparent when we apply the same process to a not so tractable by the Fourier Transform PDE:
$$
u_t=u_{xx} - \varepsilon\, x^2u_{xxxx}.
$$
To the best of our knowledge, this is the first time that such kind of analysis is performed for this type of anomalous diffusion models. Unlike previous studies that rely on numerical schemes or integral representations, our symmetry-based approach yielded an explicit analytical expression that separates the primary Fickian flux from the emergent secondary flux. Moreover, this is the first time that the two specific stabilizations of the Lie algebra of the heat equation are reported. In future works, we intent to extend this process to more complicate, and more realistic, forms of the governing equation for diffusion with retention.

\section*{CRediT authorship contribution statement}
\noindent{\bf Y.B.:} Funding Acquisition, Validation (supporting), Writing -- Original Draft (supporting), Writing -- Review \& Editing (equal). {\bf S.D.:} Conceptualization, Methodology, Validation (lead), Formal Analysis, Investigation, Visualization, Writing -- Original Draft (lead), Writing -- Review \& Editing (equal). {\bf A.J.S.N.:} Model Supervision, Writing--Review \& Editing (equal). All authors reviewed and approved the final manuscript.

\section*{Declaration of Competing Interest}
The authors declare that they have no known competing financial interests or personal relationships that could have appeared to influence the work reported in this paper.

\section*{Acknowledgements}

Y.~Bozhkov and S.~Dimas gratefully acknowledge support from the São Paulo Research Foundation-FAPESP, Grant No. 2024/02615-7. A.~J. Silva Neto thanks CAPES, CNPq and FAPERJ for their support. The readability of the text was further improved with the help of the Perplexity AI. 

\printbibliography

@article{JiaBevNet2018a,
	author = {Jiang, Maosheng and Bevilacqua, Luiz and Neto, Antonio J Silva and Gale{\~a}o, Augusto CN Rodrigues and Zhu, Jiang},
	doi = {10.1016/j.apm.2018.07.022},
	journal = {Applied Mathematical Modelling},
	pages = {121--134},
	publisher = {Elsevier},
	title = {Bi-flux theory applied to the dispersion of particles in anisotropic substratum},
	volume = {64},
	year = {2018}}

@article{BevJiaSil2016a,
	author = {Bevilacqua, L and Jiang, M and Silva Neto, A and Gale{\~a}o, ACRN},
	doi = {10.1007/s40430-015-0475-5},
	journal = {Journal of the Brazilian Society of Mechanical Sciences and Engineering},
	number = {5},
	pages = {1421--1432},
	publisher = {Springer},
	title = {An evolutionary model of bi-flux diffusion processes},
	volume = {38},
	year = {2016}}

@article{BevGalCos2011b,
	author = {Bevilacqua, Luiz and Gale{\~a}o, Augusto CNR and Costa, Flavio P},
	doi = {10.1590/S1678-58782011000200007},
	journal = {Journal of the Brazilian Society of Mechanical Sciences and Engineering},
	pages = {166--175},
	publisher = {SciELO Brasil},
	title = {On the significance of higher order differential terms in diffusion processes},
	volume = {33},
	year = {2011}}

@article{BevGalCos2011a,
	author = {Bevilacqua, Luiz and Gale{\~a}o, Augusto CNR and Costa, Flavio P},
	doi = {10.1590/S0001-37652011005000033},
	journal = {Anais da Academia Brasileira de Ci{\^e}ncias},
	pages = {1443--1464},
	publisher = {SciELO Brasil},
	title = {A new analytical formulation of retention effects on particle diffusion processes},
	volume = {83},
	year = {2011}}

@book{Ibr2009b,
	author = {Ibragimov, Nail H},
	doi = { 10.1142/7573},
	publisher = {World Scientific},
	title = {Practical Course in Differential Equations and Mathematical Modelling, A: Classical and New Methods. Nonlinear Mathematical Models. Symmetry and Invariance Principles},
	year = {2009}}

@article{TarChe2021a,
	author = {Tarayrah, Mahmood R and Cheviakov, Alexei F},
	doi = { 10.3390/sym13091612},
	journal = {Symmetry},
	number = {9},
	pages = {1612},
	publisher = {MDPI},
	title = {Relationship between unstable point symmetries and higher-order approximate symmetries of differential equations with a small parameter},
	volume = {13},
	year = {2021}}

@book{Ibr2008a,
	address = {Blekinge Institute of Technology, Karlskrona, Sweden},
	author = {Ibragimov, N. H.},
	publisher = {ALGA publications},
	title = {Approximate {T}ransformation {G}roups},
	year = {2008}}

@misc{MathematicaV14,
	author = {Inc., Wolfram Research{,}},
	note = {Champaign, IL, 2024},
	title = {Mathematica, {V}ersion 14.x},
	url = {https://www.wolfram.com/mathematica}}

@book{Ibra95c,
	address = {Boca Raton, Fl.},
	author = {Ibragimov, N. H.},
	doi = {10.1201/9781003419808},
	edition = {1$^{st}$ edition},
	publisher = {CRC Press},
	title = {CRC Handbook of Lie Group Analysis of Differential Equations},
	volume = {1 -- 3},
	year = {1995}}

@book{Olver2k,
	address = {New York},
	author = {Olver, P. J.},
	doi = {10.1007/978-1-4612-4350-2},
	edition = {2$^{nd}$ edition},
	publisher = {Springer},
	series = {Graduate Texts in Mathematics},
	title = {Applications of Lie Groups to Differential Equations},
	volume = {107},
	year = {2000}}

@book{Hydon2k,
	address = {Cambridge},
	author = {Hydon, P. E.},
	doi = { 10.1017/CBO9780511623967},
	edition = {1$^{st}$ edition},
	publisher = {Cambridge University Press},
	series = {Cambridge Texts in Applied Mathematics},
	title = {Symmetry Methods for Differential Equations},
	year = {2000}}

@book{Ste90,
	address = {Cambridge},
	author = {Stephani, H.},
	doi = { 10.1017/CBO9780511599941},
	edition = {1$^{st}$ edition},
	note = {Editor: MacCallum, Malcolm},
	publisher = {Cambridge University Press},
	title = {Differential Equations: Their Solution Using Symmetries},
	year = {1990}}

@book{Ibra85a,
	author = {Ibragimov, N. H.},
	publisher = {D. Reidel Publishing Co.},
	title = {Transformation groups applied to mathematical physics (Translated from the Russian Mathematics and its Applications (Soviet Series))},
	year = {1985}}

@book{Ovsi82,
	author = {Ovsiannikov, L. V.},
	doi = {10.1016/C2013-0-07470-1},
	edition = {1$^{st}$ edition},
	month = {06},
	note = {432 pages},
	publisher = {Academic Press},
	title = {Group Analysis of Differential Equations},
	year = {1982}}

@book{BluKu89,
	address = {New York},
	author = {Bluman, G. W. and Kumei, S.},
	publisher = {Springer},
	title = {Symmetries and differential equations},
	year = {1989}}

@inproceedings{DiTs2k5a,
	address = {Nicosia},
	author = {Dimas, S. and Tsoubelis, D.},
	booktitle = {The 10$^{th}$ International Conference in {MO}dern {GR}oup {AN}alysis},
	editor = {Ibragimov, N.H. and Sophocleous, C. and Damianou, P.A.},
	pages = {64-70},
	publisher = {University of Cyprus},
	title = {{SYM}: A new symmetry-finding package for {M}athematica},
	year = {2005}}

\end{document}